\documentclass{amsart}

\usepackage{lmodern}
\usepackage[T1]{fontenc}
\usepackage[utf8]{inputenc}
\usepackage{hyperref}
\hypersetup{
  pdfauthor={Zijia Li, Josef Schicho, Hans-Peter Schröcker},
  pdftitle={The Geometry of Quadratic Quaternion Polynomials in Euclidean and Non-Euclidean Planes},
  pdfkeywords={},
  hidelinks,
}

\usepackage{amsmath,amsfonts,amssymb}
\usepackage{mathtools}

\renewcommand{\H}{\mathbb{H}}
\newcommand{\R}{\mathbb{R}}
\renewcommand{\S}{\mathbb{S}}
\newcommand{\cj}[1]{\overline{#1}}
\newcommand{\ci}{\mathrm{i}}
\newcommand{\qi}{\mathbf{i}}
\newcommand{\qj}{\mathbf{j}}
\newcommand{\qk}{\mathbf{k}}
\newcommand{\IP}[2]{\langle #1, #2 \rangle}
\newcommand{\CP}[2]{#1 \times #2}
\newcommand{\NC}{\mathcal{N}}
\newcommand{\IC}{\mathcal{S}}
\newcommand{\Quadrance}[2]{q(#1,#2)}

\newcommand{\FourBar}[4]{\mathcal{F}(#1,#2;#3,#4)}

\newtheorem{thm}{Theorem}

\newtheorem{cor}{Corollary}
\theoremstyle{definition}
\newtheorem{defn}{Definition}
\theoremstyle{remark}
\newtheorem{rmk}{Remark}
\newtheorem{example}{Example}

\title[The Geometry of Quadratic Quaternion Polynomials \ldots]{The Geometry of Quadratic Quaternion Polynomials in Euclidean and Non-Euclidean Planes}
\date{\today}

\author{Zijia Li}
\address[Zijia Li]{Institute of Discrete Mathematics and Geometry, Vienna University of Technology, Wiedner Hauptstrasse 8-10/104, 1040 Vienna, Austria}
\email{zijia.li@tuwien.ac.at}

\author{Josef Schicho}
\address[Josef Schicho]{Research Institute for Symbolic Computation, Johannes Kepler University, Linz, Austria}
\email{josef.schicho@risc.jku.at}
\author{Hans-Peter Schröcker}
\address[Hans-Peter Schröcker]{Unit Geometry and CAD, University of Innsbruck, Technikerstr.~13, 6020 Innsbruck, Austria}
\email{hans-peter.schroecker@uibk.ac.at}


\begin{document}

\begin{abstract}
    We propose a geometric explanation for the observation that generic quadratic
  polynomials over split quaternions may have up to six different factorizations
  while generic polynomials over Hamiltonian quaternions only have two. Split
  quaternion polynomials of degree two are related to the coupler motion of
  ``four-bar linkages'' with equal opposite sides in universal hyperbolic
  geometry. A factorization corresponds to a leg of the four-bar linkage and
  during the motion the legs intersect in points of a conic whose focal points
  are the fixed revolute joints. The number of factorizations is related by the
  number of real focal points which can, indeed, be six in universal hyperbolic
  geometry. \end{abstract}

\maketitle

\section{Introduction}

The aim of this paper is a geometric explanation for a purely algebraic
observation on factorizations of quaternion polynomials. While generic quadratic
polynomials over the (Hamiltonian) quaternions admit two different
factorizations into linear factors, there exist up to six such factorizations
for generic polynomials over split quaternions \cite{li18}. What is not too
surprising from a purely algebraic viewpoint has rather strange consequences in
the kinematics of the hyperbolic plane. Polynomials over Hamiltonian quaternions
parameterize rational spherical motions; factorizations correspond to the
decomposition into sequences of coupled rotations and give rise to mechanisms
one of whose links follows this motion. In particular, the two factorizations of
a quadratic polynomial correspond to the two legs of a spherical four-bar
linkage. Similar constructions are possible in the hyperbolic plane but then
give rise to ``four-bar linkages'' with up to \emph{six} legs -- even in generic
cases!

In this article, we briefly explain the algebraic background of quaternion
polynomial factorization and then provide a geometric interpretation for above
phenomena in universal hyperbolic geometry according to
\cite{wildberger10,wildberger11,wildberger13}. The four-bar linkages in question
have equal opposite sides and are related to certain conics whose focal points
are the fixed and the moving revolute joints, respectively. It is rather obvious
that a conic in universal hyperbolic geometry can have up to six real focal
points and this corresponds to the number of six possible factorizations. We
also find other relations between the geometry of this configuration and the
algebra of split quaternion polynomials. An observation which is geometrically
evident but not obvious from a purely algebraic viewpoint is that the six fixed
and the six moving joints are vertices of complete quadrilaterals.

Our geometric interpretation is inspired by \cite{schoger96} which treats
four-bar linkages with equal opposite sides in traditional hyperbolic geometry.
Similar properties of four-bar linkages with equal opposite sides in elliptic
or Euclidean geometry are known as well.

\section{Quaternions and Quaternion Polynomials}

The \emph{Hamiltonian quaternions $\H$} form a real associative algebra of
dimension four. An element $h \in \H$ is written as $h = h_0 + h_1\qi + h_2\qj +
h_3\qk$ with $h_0$, $h_1$, $h_2$, $h_3 \in \R$. The multiplication rules can be
derived from the relations
\begin{equation}
  \label{eq:1}
  \qi^2 = \qj^2 = \qk^2 = \qi\qj\qk = -1.
\end{equation}
Changing some signs in \eqref{eq:1} gives the multiplication rules for the
\emph{split quaternions~$\S$:}
\begin{equation}
  \label{eq:2}
  \qi^2 = -\qj^2 = -\qk^2 = -\qi\qj\qk = -1.
\end{equation}
In this text, we focus on split quaternions. It parallels, in large parts, the
Hamiltonian theory but there are also important differences which we will
emphasize whenever appropriate.

The conjugate quaternion is $\cj{h} \coloneqq h_0 - h_1\qi - h_2\qj - h_3\qk$,
the quaternion norm is $h\cj{h}$. Carrying out the computation, we find $h\cj{h}
= h_0^2 + h_1^2 - h_2^2 - h_3^2 \in \R$. The quaternion $h$ has a multiplicative
inverse if and only if $h\cj{h} \neq 0$. The inverse is then given as
\begin{equation}
  \label{eq:3}
  h^{-1} = (h\cj{h})^{-1}\cj{h}.
\end{equation}

The split quaternions $\S$ form a ring with center $\R$ while the Hamiltonian
quaternions even form a division ring because their norm $h\cj{h} = h_0^2 +
h_1^2 + h_2^2 + h_3^2$ is positive unless $h = 0$.

Denote by $\S[t]$ the ring of polynomials in the indeterminate $t$ with
coefficients in $\S$ where multiplication is defined by the convention that $t$
commutes with all coefficients. Given $C = \sum_{\ell=0}^n c_\ell t^\ell \in
\S[t]$, the \emph{conjugate polynomial} is defined as $\cj{C} \coloneqq
\sum_{\ell=0}^n \cj{c_\ell}t^\ell$. This implies that $C\cj{C}$ is real (and
non-negative in the Hamiltonian case).

We say that $C = \sum_{\ell=0}^n c_\ell t^\ell \in \S[t]$ admits a factorization
if there exist $h_1$, $h_2,\ldots, h_n \in \S$ such that
\begin{equation*}
  C = c_n(t-h_1)(t-h_2) \cdots (t-h_n).
\end{equation*}

In this text, we will confine ourselves to \emph{generic} cases only. This
allows us to assume that the leading coefficient $c_n$ is invertible. As far as
factorizability is concerned, it is then no loss of generality to assume $c_n =
1$ (monic polynomials). Other generic properties are:
\begin{itemize}
\item The coefficients of $C$ are linearly independent.
\item The linear polynomial $R = C - M$ has a unique zero for every real monic
  quadratic factor $M$ of $C\cj{C}$.
\end{itemize}

There is an algorithm for factorizing quaternion polynomials based on the
factorization of the norm polynomial $C\cj{C}$ over $\R$
\cite{hegedus13,li15,li15b,li18}. This algorithm will work for generic
polynomials but may fail in special circumstances. We are going to describe it
for generic monic polynomials of degree two. An extension to higher degrees is
possible but requires a suitable concepts of polynomial division which we do not
need here.

In order to compute a factorization of a generic quadratic polynomial $C$, pick
a monic quadratic divisor $M$ of $C\cj{C}$. The polynomial $R \coloneqq C - M$
is linear and we may write $R = r_1t + r_0$ with $r_0$, $r_1 \in \S$. By
assumption, the leading coefficient $r_1$ of $R$ is invertible whence we may set
$h_2 \coloneqq -r_1^{-1}r_0$. Theory ensures that $t-h_2$ is a right factor of
$C$, that is, there exists $h_1 \in \S$ such that $C = (t-h_1)(t-h_2)$, see
\cite{li18}. Comparing coefficients of $C = t^2 + c_1t + c_0$ and
$(t-h_1)(t-h_2) = t^2 - (h_1+h_2)t + h_1h_2$ we find $h_1 = -c_1 - h_2$.

Two observations are crucial to us:
\begin{itemize}
\item In non-generic cases, the leading coefficient of $r_1$ may fail to be
  invertible. In this case, neither existence nor uniqueness of $h_2$ can be
  guaranteed.
\item Unless $C\cj{C}$ is the square of a quadratic polynomial, it has at least
  one more monic quadratic factor $N$ which gives rise to at least one more
  factorization $C = (t-k_1)(t-k_2)$ with $h_1 \neq k_1$, $h_2 \neq k_2$.
\end{itemize}

\begin{defn}
  Two factorizations $C = (t-h_1)(t-h_2) = (t-k_1)(t-k_2)$ of a generic
  quadratic polynomial $C$ are called \emph{complementary} if $C\cj{C} =
  (t-h_1)(t-\cj{h_1})(t-k_1)(t-\cj{k_1})$.
\end{defn}

The defining condition for complementary factorization is equivalent to $C\cj{C}
= (t-h_2)(t-\cj{h_2})(t-k_2)(t-\cj{k_2})$. The meaning of this concept is that
the two factorizations are obtained from relatively prime quadratic factors $M$,
$N$ of $C\cj{C} = MN$. Factorizations in the Hamiltonian case are always
complementary.

Since we will need it later, we explicitly compute the complementary
factorization of $C = (t-h_1)(t-h_2)$. We set $M \coloneqq (t-h_1)(t-\cj{h_1})$
and obtain $C = M + R$ where $R = -(h_2-\cj{h_1})t + h_1(h_2-\cj{h_1})$. The
sought quaternion $k_2$ of the complementary factorization is the unique zero
of~$R$:
\begin{equation}
  \label{eq:4}
  k_2 = (\cj{h_2}-h_1)^{-1}h_1(\cj{h_2}-h_1).
\end{equation}
Comparing coefficients of
\begin{equation}
  \label{eq:5}
  (t-h_1)(t-h_2)=(t-k_1)(t-k_2),
\end{equation}
the value of $k_1$ can be computed as $k_1 = h_1 + h_2 - k_2$. We aim, however,
at a formula similar to \eqref{eq:4}. Taking conjugates on both sides of
\eqref{eq:5}, we arrive at $(t-\cj{h_2})(t-\cj{h_1}) =
(t-\cj{k_2})(t-\cj{k_1})$, whence $\cj{k_1} =
(h_1-\cj{h_2})^{-1}\cj{h_2}(h_1-\cj{h_2})$, or, after conjugating, $k_1 =
(\cj{h_1}-h_2) h_2 (\cj{h_1}-h_2)^{-1}$. Considering the relation \eqref{eq:2}
between inverse and conjugate quaternion, this finally gives
\begin{equation}
  \label{eq:6}
  k_1 = (\cj{h_2}-h_1)^{-1} h_2 (\cj{h_2}-h_1).
\end{equation}
Note the symmetry between \eqref{eq:4} and \eqref{eq:6}.

We conclude this section with two examples concerning the number of
factorizations in the Hamiltonian and in the split quaternion case. In the
Hamiltonian case, at most two factorizations do exist while more factorizations
are possible, even for \emph{generic} split quaternion polynomials.

\begin{example}
  The polynomial $C = t^2 - (2+\qj+2\qk)t + 1 + 2\qi + \qj + 2\qk \in \H[t]$
  admits the two factorizations
  \begin{equation*}
    C = (t - 1 - \qj)(t - 1 - 2\qk)
    = (t - 1 - \tfrac{8}{5}\qj - \tfrac{6}{5}\qk)(t -1 + \tfrac{3}{5}\qj - \tfrac{4}{5}\qk).
 \end{equation*}
  Other factorizations do not exist. This is a generic case and factorizations
  can be computed as described above. Note that $C\cj{C} = M_1M_2$ with $M_1 =
  t^2 - 2t + 2$ and $M_2 = t^2-2t + 5$. Other monic quadratic factors of
  $C\cj{C}$ do not exist.
\end{example}

\begin{example}
  The polynomial $C = t^2 - (2 + \qj + 2\qk)t + 1 - 2\qi + \qj + 2\qk \in \S[t]$
  over the split quaternions admits precisely six different factorizations (two
  of which are formally identical to the factorizations of the previous example):
  \begin{equation*}
    \begin{aligned}
      C &= (t - 1 - \qj)(t - 1 - 2\qk) \\
        &= (t - 1 - \tfrac{8}{5}\qj - \tfrac{6}{5}\qk)(t - 1 + \tfrac{3}{5}\qj - \tfrac{4}{5}\qk) \\
        &= (t + \tfrac{1}{2} + \tfrac{3}{2}\qi + \tfrac{1}{2}\qj - \tfrac{3}{2}\qk )(t -\tfrac{5}{2} - \tfrac{3}{2}\qi - \tfrac{3}{2}\qj - \tfrac{1}{2}\qk) \\
        &= (t - \tfrac{5}{2} - \tfrac{3}{2}\qi + \tfrac{1}{2}\qj - \tfrac{3}{2}\qk)(t + \tfrac{1}{2} + \tfrac{3}{2}\qi - \tfrac{3}{2}\qj - \tfrac{1}{2}\qk) \\
        &= (t - \tfrac{1}{2} - \tfrac{1}{2}\qi - \tfrac{3}{2}\qj - \tfrac{1}{2}\qk)(t - \tfrac{3}{2} + \tfrac{1}{2}\qi + \tfrac{1}{2}\qj - \tfrac{3}{2}\qk) \\
        &= (t - \tfrac{3}{2} + \tfrac{1}{2}\qi - \tfrac{3}{2}\qj - \tfrac{1}{2}\qk)(t - \tfrac{1}{2} - \tfrac{1}{2}\qi + \tfrac{1}{2}\qj - \tfrac{3}{2}\qk).
    \end{aligned}
  \end{equation*}
  The number of six factorizations is explained by the real factorization
  $C\cj{C} = t(t+1)(t-2)(t-3)$ of the norm polynomial. There exist six real
  quadratic factors of $C\cj{C}$, each of them giving rise to a different
  factorization of~$C$.
\end{example}

\section{Universal Hyperbolic Geometry}

We are going to provide a concise introduction to hyperbolic geometry via split
quaternions. The vector part of a quaternion $h$ is $\tfrac{1}{2}(h - \cj{h})$.
Quaternions that equal their vector part are called \emph{vectorial.} In the
vector space $V$ of vectorial split quaternions we define \emph{inner product}
and \emph{cross product} according to $\IP{a}{b} \coloneqq \frac{1}{2}(a\cj{b} +
b\cj{a})$ and $\CP{a}{b} \coloneqq \frac{1}{2}(ab - ba)$. Denote by $P^2$ the
projective plane over $V$. The quadratic form $x \mapsto \IP{x}{x}$ defines a
regular conic $\NC \subset P^2$ with real points. It is the absolute circle of a
hyperbolic geometry. For our purposes it will be advantageous to adopt the
viewpoint of \emph{universal hyperbolic geometry} in the sense of
\cite{wildberger10,wildberger11,wildberger13}. Hence, we also refer to $\NC$ as
\emph{null circle.} The points of $\NC$ are called \emph{null points,} the
tangents of $\NC$ are called \emph{null lines.}

We represent the straight line spanned by two points $[a]$ and $[b]$ by
$[\CP{a}{b}]$. A line $[u]$ and a point $[x]$ are incident if and only if
$\IP{u}{x} = 0$. The \emph{quadrance} between two non-null points $[a]$ and
$[b]$ is defined as
\begin{equation*}
  \Quadrance{[a]}{[b]} \coloneqq 1 - \frac{\IP{a}{b}^2}{\IP{a}{a}\IP{b}{b}}.
\end{equation*}
Quadrances correspond to squared distances in traditional hyperbolic geometry.
Because of the absence of transcendential functions in their definition, they
better fit into our algebraic framework. Points on a common null line are
characterized by zero quadrance. The reflection in a non-null line $z$ is the
unique homology that fixes that line and the absolute conic $\NC$. It also fixes
the pole $Z$ of $z$ with respect to $\NC$ and hence is referred to as reflection
in $Z$ as well \cite{wildberger10}.

Using ordinary quaternions instead of split quaternions results in a similar
algebraic description of planar elliptic geometry. Here, the null conic has
index zero whence null points and null lines only exist in the complex extension
of the real projective plane.

\section{The Kinematics of Polynomial Factorization}

Factorization of quaternion polynomials is closely related to planar kinematics.
A quaternion $h \in \S$ acts on a point $[x] \in P^2$ via $[x] \mapsto
[hx\cj{h}]$. Now the straightforward computation
\begin{equation*}
  2\IP{hx\cj{h}}{hy\cj{h}} =
  hx\cj{h}\cj{(hy\cj{h})} + hy\cj{h}\cj{(h{x}\cj{h})} =
  h\cj{h}\;h(x\cj{y} + y\cj{x})\cj{h} =
  (h\cj{h})^2\IP{x}{y}
\end{equation*}
shows that
\begin{equation*}
  \Quadrance{[hx_1\cj{h}]}{[hx_2\cj{h}]} =
  \Quadrance{[x_1]}{[x_2]}
\end{equation*}
holds for any two points $[x_1]$, $[x_2] \in P^2$. Hence, the action of $h$ is
an isometry of the underlying metric geometry. It has the real fix point
$[h-\cj{h}]$ and two further fix points -- the null points on the absolute polar
of $[h-\cj{h}]$. They are always complex in the Hamiltonian case (elliptic
geometry) but may be real in the split quaternion case (hyperbolic geometry). We
call any map of above type a \emph{rotation} and the point $[h-\cj{h}]$ its
\emph{center.} Sometimes we will simply speak of ``the rotation~$h$''.

The center $[h-\cj{h}]$ of $t-h$ is independent of $t$. Hence, linear
polynomials parameterize rotations with a fixed centers and the factorization of
a polynomial $C$ corresponds to the decomposition of the motion parameterized by
$C$ into a sequence of coupled rotations. This gives rise to a mechanical
interpretation of the factorizations $C = (t-h_1)(t-h_2) = (t-k_1)(t-k_2)$. For
the time being, we assume that these factorizations are complementary. (The case
of non-complementary factorizations will be clarified later, see
Corollaries~\ref{cor:1} and~\ref{cor:2}.)

Consider the four-bar linkage $\FourBar{H_1}{K_1}{K_2}{H_2}$ with fixed revolute
joints at $H_1 \coloneqq [h_1-\cj{h_1}]$, $K_1 \coloneqq [k_1-\cj{k_1}]$ and
corresponding moving revolute joints at $H_2 \coloneqq [h_2-\cj{h_2}]$, $K_2
\coloneqq [k_2-\cj{k_2}]$. The polynomial $C$ describes a motion that can be
performed by a system rigidly connected to the moving point pair $(H_2,K_2)$ in
the four-bar linkage. Comparing \eqref{eq:2} and \eqref{eq:3} we see that this
four-bar linkage is rather special: The pair $(k_1,k_2)$ is obtained from
$(h_2,h_1)$ by left and right multiplying with a fixed quaternion. This amounts
to a mere change of coordinates and shows that the kinematics and geometry of
these two pairs are completely identical. In particular, we obtain

\begin{thm}
  \label{th:1}
  Opposite sides in a four-bar linkage $\FourBar{H_1}{K_1}{K_2}{H_2}$ obtained
  from two complementary factorizations $C = (t-h_1)(t-h_2) =
  (t-k_1)(t-k_2)$have equal quadrances: $\Quadrance{H_1}{H_2} =
  \Quadrance{K_1}{K_2}$ and $\Quadrance{H_1}{K_1} = \Quadrance{H_2}{K_2}$.
\end{thm}

Four-bar linkages of the type described in Theorem~\ref{th:1} posses two folded
configurations where all revolute joints are collinear and two ``motion modes''
(corresponding to irreducible components of the motion in the projectivized
space of quaternions). In Euclidean geometry, they are called ``parallelogram''
and ``anti-parallelogram mode'', respectively -- a terminology that does not
make sense in elliptic or hyperbolic geometry and is not necessary either.
(Although one may attempt to recover some aspects of the Euclidean situation in
traditional hyperbolic geometry as done in \cite{schoger96}.) The polynomial $C$
only describes one motion component of $\FourBar{H_1}{K_1}{K_2}{H_2}$. In
Euclidean or spherical geometry, this is well-known; in hyperbolic geometry this
is not difficult to show.

\section{A Conic and its Focal Points}

In this section we investigate more closely the geometry of the four-bar linkage
$\FourBar{H_1}{K_1}{K_2}{H_2}$ obtained from two complementary factorizations $C
= (t-h_1)(t-h_2) = (t-k_1)(t-k_2)$. Denote by $H_2(t)$ and $K_2(t)$ the position
of the moving joints at parameter value $t$ and by $S(t)$ the intersection point
of the lines spanned by $H_1$, $H_2(t)$ and $K_1$, $K_2(t)$, respectively
whenever this point is well defined. This is not the case in folded positions
but then $S(t)$ can be defined by continuity.

Recall the following definition:

\begin{defn}
  \label{def:focal-point}
  The focal points of a conic in hyperbolic geometry are the real intersection
  points of its null tangents.
\end{defn}

Depending on the number of real null tangents, a regular conic may have between
two and six focal points.

\begin{thm}
  \label{th:2}
  The locus of all points $S(t)$ for varying parameter $t$ is a conic
  $\mathcal{S}$ with focal points $H_1$ and~$K_1$. If $\varrho(t)$ denotes the
  reflection in the conic tangent at $S(t)$, then $K_2(t) = \varrho(H_1)$ and
  $H_2(t) = \varrho(K_1)$.
\end{thm}

\begin{proof}
  Since we only consider generic cases in this paper, we can assume that one
  fixed revolute joint, say $H_1$, is contained in the interior or in the
  exterior of $\NC$. Our proof will consist of a straightforward computation
  based on a case distinction between these two cases.

  Assume at first that $H_1$ is contained in the interior of $\NC$. There exists
  a suitable isometry of hyperbolic geometry that maps $H_1$ to $[1]$. Moreover,
  we pick an initial position of the four-bar linkage where $H_2$ lies on the
  line spanned by $[\qi]$ and $[\qj]$ and apply a suitable linear
  re-parametrization that ensures $h_2 + \cj{h_2} = 0$. This allows us to write
  $C = (t-h_1)(t-h_2)$ where
  \begin{equation*}
    h_1 = h_{10} + h_{11}\qi
    \quad\text{and}\quad
    h_2 = h_{21}\qi + h_{22}\qj.
  \end{equation*}
  Using Equations~\eqref{eq:4} and \eqref{eq:6} we can compute the complementary
  factorization:
  \begin{gather*}
    \begin{aligned}
      k_1 &= N\bigl((h_{10}^2h_{21}+h_{11}^2h_{21}+2h_{11}h_{21}^2-2h_{11}h_{22}^2+h_{21}^3-h_{21}h_{22}^2)\qi
      \\ &\qquad +h_{22}(h_{10}^2-h_{11} ^2+h_{21}^2-h_{22}^2)\qj + 2h_{10}h_{11}h_{22}\qk \bigr),\\
      k_2 &=
      N \bigl(
      h_{11}(h_{10}^2+h_{11}^2+2h_{11}h_{21}+h_{21}^2+h_{22}^2)\qi \\ &\qquad +2h_{11}h_{22}(h_{11}+h_{21})\qj -2h_{10}h_{11}h_{22}\qk 
      \bigr)
      + h_{10}.
    \end{aligned}
  \end{gather*}
  with $N = (h_{10}^2+h_{11}^2+2h_{11}h_{21}+h_{21}^2-h_{22}^2)^{-1}$.

  The fixed revolute joints are $H_1 \coloneqq [h_1 - \cj{h_1}]$ and $K_1
  \coloneqq [k_1 - \cj{k_1}]$; the paths of the moving joints are parameterized
  by $H_2(t) = [\eta_2(t)]$, $K_2(t) = [\varkappa_2(t)]$ with
  \begin{equation*}
    \eta_2(t) \coloneqq (t-h_1)(h_2-\cj{h_2})(t-\cj{h_1}),\quad
    \varkappa_2(t) \coloneqq (t-k_1)(k_2-\cj{k_2})(t-\cj{k_1}).
  \end{equation*}
  This gives $S(t) = [\sigma(t)]$ where $\sigma(t) = (h_1 \times \eta_2(t))
  \times (k_1 \times \varkappa_2(t)) = FG$ and
  \begin{gather*}
    F = 2h_{11}^2h_{22}^2(h_{10}t^2 -(h_{10}^2 + h_{11}^2 - h_{21}^2 + h_{22}^2)t -h_{10}(h_{21}^2 - h_{22}^2)
    ),\\
    G =
    -((h_{11}+h_{21})\qi+h_{22}\qj)t^2
    +(2h_{10}h_{21}\qi+2h_{10}h_{22}\qj+2h_{11}h_{22}\qk)t\\
    + (h_{11}(h_{22}^2-h_{21}^2) -h_{21}(h_{10}^2+h_{11}^2))\qi
    + h_{22}(h_{11}^2-h_{10}^2)\qj
    - 2h_{10} h_{11}h_{22}\qk).
  \end{gather*}
  The homogeneous parametric equation of $S(t)$ is quadratic whence its locus
  is, indeed, a conic $\IC$. We still have to show that $H_1$ and $K_1$ are each
  incident with two null tangents of $\IC$. Their parameter values are solutions
  of the quartic polynomial $(G \times \frac{\mathrm{d}}{\mathrm{d}t}G)\cj{(G
    \times \frac{\mathrm{d}}{\mathrm{d}t}G)}$. The solution set equals
  \begin{equation*}
    \{\pm(h_{22}^2-h_{21}^2)^{1/2}, h_{10} \pm \ci h_{11} \}
  \end{equation*}
  and it can readily be verified that the first two tangents intersect in $K_1$
  while the last two intersect in $H_1$. The statement on the reflection admits
  a straightforward computational proof as well. In traditional hyperbolic
  geometry it is actually well-known \cite{schoger96}.

  If $H_1$ is contained in the exterior of $\NC$ we may assume
  \begin{equation*}
    h_1 = h_{10} + h_{12}\qj,\quad
    h_2 = h_{21}\qi + h_{22}\qj.
  \end{equation*}
  A similar computation then yields the parameter values
  \begin{equation}
    \label{eq:7}
    \{\pm(h_{22}^2-h_{21}^2)^{1/2}, h_{10}\pm h_{12} \}
  \end{equation}
  for the null tangents. Again, $K_1$ is incident with the first pair of
  null tangents and $H_1$ with the last pair.
  \hfill\qed
\end{proof}

An immediate consequence of Theorem~\ref{th:2} is the following characterization
of the geometry of fixed and moving revolute joints in four-bar linkages coming
from factorizations of a generic quadratic polynomial~$C$.

\begin{cor}
  \label{cor:1}
  If a generic quadratic motion polynomial $C$ admits six factorizations $C =
  (t-a_\ell)(t-b_\ell)$, $\ell \in \{1,2,3,4,5,6\}$, the six points $A_{\ell}
  \coloneqq [a_\ell-\cj{a_\ell}]$ are vertices of a complete quadrilateral whose
  sides are null. Opposite vertices correspond to complementary factorizations.
  Similar statements hold true for the six points $B_{\ell} \coloneqq
  [b_\ell-\cj{b_\ell}]$.
\end{cor}

\begin{proof}
  By Definition~\ref{def:focal-point} and Theorem~\ref{th:2} the points $A_\ell$
  are the intersection points of the common tangents of the null conic $\NC$ and
  the conic $\IC$. The statement on the points $B_\ell$ follows by considering
  the inverse motion which is parameterized by $\cj{C}$.\hfill\qed
\end{proof}

\begin{rmk}
  As a consequence of Corollary~\ref{cor:1}, four-bar linkages with six real
  legs do not exist in traditional hyperbolic geometry because the fixed and
  moving vertices are necessarily in the exterior of $\NC$.
\end{rmk}

\begin{rmk}
  In the kinematics of traditional hyperbolic geometry it is known that there
  are up to six points whose trajectories have fourth order contact with their
  respective curvature circles. These points are called \emph{Burmester points}
  \cite{inalcik17}. If a generic quadratic motion polynomial admits six
  factorizations, it parameterizes a motion whose six real Burmester points are
  the moving revolute joints.
\end{rmk}

Each quaternion $b_\ell$ is the root of a linear polynomial $C-M_\ell$ where
$M_\ell$ is a monic quadratic factor of $C\cj{C}$. We may label them by elements
of the set $\{\{1,2\}, \{1,3\}, \{1,4\}, \{2,3\}, \{2,4\}, \{3,4\}\}$ in such a
way that two labels share an element if and only if the corresponding quadratic
polynomials have a linear factor in common. Call these polynomials and also the
corresponding vertices \emph{linked.} The labeling of quadratic polynomials
extends to the moving joints $B_\ell$ and, by the convention that vertices that
correspond in the reflection $\varrho(t)$ have disjoint labels, also to the
fixed vertices $A_\ell$. In this way it is guaranteed that the quadrance of
vertices $A_{rs}$ and $B_{rs}$ remains constant during the motion.

\begin{figure}
  \centering
  \includegraphics{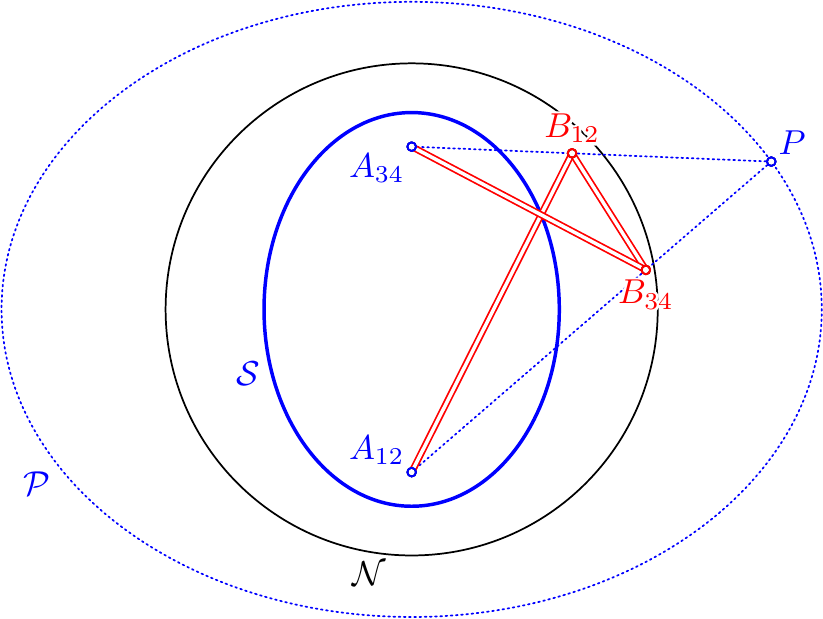}
  \caption{Geometry of fixed and moving revolute joints (two focal points)}
  \label{fig:hfb-2f}
\end{figure}

\begin{figure}
  \centering
  \includegraphics{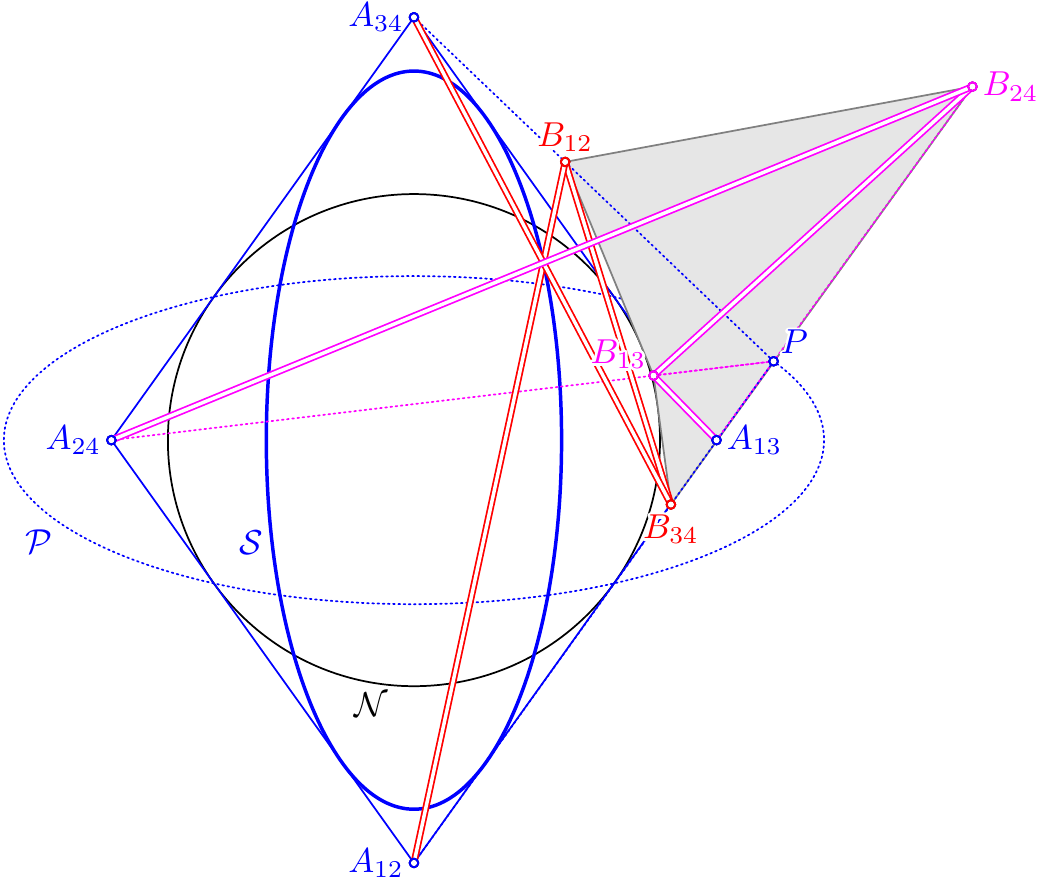}
  \caption{Geometry of fixed and moving revolute joints (six focal points)}
  \label{fig:hfb-6f}
\end{figure}

Figures~\ref{fig:hfb-2f} and \ref{fig:hfb-6f} illustrate the geometric relation
between the fixed revolute joints $A_{ij}$ and the moving revolute joints
$B_{ij}$. The first figure displays the case of only two focal points $A_{12}$
and $A_{34}$. We see the conic $\IC$ and the absolute polar conic $\mathcal{P}$
of its dual. There is a hyperbolic reflection with center $P \in \mathcal{P}$
that maps $A_{12}$ to $B_{12}$ and $A_{34}$ to $B_{34}$. As $P$ varies on
$\mathcal{P}$, the four-bar linkage moves.

The situation in Figure~\ref{fig:hfb-6f} is similar but the conic $\IC$ has six
focal points. Two of them are at (Euclidean) infinity and are not displayed. We
essentially repeated the construction of Figure~\ref{fig:hfb-2f} twice, once for
the focal points $A_{12}$, $A_{34}$ and once for the focal points
$A_{13}$, $A_{24}$. We would like to point the readers attention to the
following relations:
\begin{itemize}
\item All connecting lines of points $A_{ij}$ and $B_{ij}$ intersect in the same
  point of~$\IC$.
\item The tangent of $\IC$ in this point is the absolute polar of~$P$.
\item The connecting lines of points $A_{ij}$ and $A_{\ell r}$ is null if the
  sets $\{i,j\}$ and $\{\ell,r\}$ have a non-empty intersection.
\item The connecting lines of points $B_{ij}$ and $B_{\ell r}$ is null if the
  sets $\{i,j\}$ and $\{\ell,r\}$ have a non-empty intersection.
\end{itemize}

Corollary~\ref{cor:1} leaves a certain ambiguity as to the position of linked
fixed/moving joints: Either they are collinear or the form a triangle with null
sides. The next corollary answers this question.

\begin{cor}
  \label{cor:2}
  Linked vertices are collinear.
\end{cor}

\begin{proof}
  Using the setup of the second case in the proof of Theorem~\ref{th:2} we have
  $C\cj{C} = (t-t_1)(t-t_2)(t-t_3)(t-t_4)$ where
  \begin{equation}
    \label{eq:8}
    t_1 = h_{10} + h_{12},\
    t_2 = h_{10} - h_{12},\
    t_3 = (h_{22}^2-h_{21}^2)^{1/2},\
    t_4 = -(h_{22}^2-h_{21}^2)^{1/2}.
  \end{equation}
  For $(i,j) \in \{ (1,2), (1,3), (1,4) \}$ we set $M_{ij} \coloneqq (t - t_i)(t
  - t_j)$ and $h_{ij} \coloneqq -r_{ij,1}^{-1}r_{ij,0}$ where $r_{ij,1}t +
  r_{ij,0} = C - M_{ij}$ and $B_{ij} \coloneqq [h_{ij} - \cj{h_{ij}}]$. We then
  have
  \begin{equation*}
    \begin{aligned}
    B_{12} &= [-4h_{12}h_{21}(h_{12}+h_{22})\qi + 2h_{12}(h_{10}^2 -h_{21}^2 - (h_{12}+h_{22})^2)\qj -4h_{10}h_{12}h_{21}\qk ],\\
    B_{13} &= [h_{21}(w+u+h_{22})\qi +((w+h_{22})u +h_{21}^2)\qj +h_{21}(w +h_{10}-h_{12})\qk],\\
    B_{14} &= [h_{21}(w -u - h_{22})\qi +((w-h_{22})u -h_{21}^2)\qj -h_{21}(w -h_{10}+h_{12})\qk]
    \end{aligned}
  \end{equation*}
  where $w = (h_{22}^2-h_{21}^2)^{1/2}$ and $u = h_{12}-h_{10}+h_{22}$. It can
  readily be verified that these points are collinear. \hfill\qed
\end{proof}

Finally, the proof of Corollary~\ref{cor:2} also demonstrates a geometric
property of the roots of the norm polynomial.

\begin{cor}
  \label{cor:3}
  If $t_0$ is a zero of $C\cj{C}$ then the tangent of $\mathcal{S}$ at $S(t_0)$
  is null.
\end{cor}

\begin{proof}
  The parameter values of null tangents are given in \eqref{eq:7}, the zeros of
  $C\cj{C}$ are given in \eqref{eq:8}. Obviously, these values are identical.
  \hfill\qed
\end{proof}

\section{Comparison with the Euclidean Case}

We already mentioned that the coupler motion of four-bar linkages with equal
opposite sides in Euclidean geometry has a parallelogram and an
anti-parallelogram mode. Using dual quaternions, a factorization theory similar
to the elliptic or hyperbolic geometry is possible
\cite{gallet17,hegedus13,li15b,li18}. However, the behavior of parallelograms
and anti-parallelograms in this context is quite different.
\begin{itemize}
\item Generic quadratic motion polynomials describe anti-parallelogram linkages
  and admit two factorizations.
\item Parallelogram linkages are described by rather special quadratic motion
  polynomials that admit \emph{infinitely} many factorizations. Each
  factorizations corresponds to one of infinitely many legs of the corresponding
  parallelogram linkage.
\end{itemize}
This is illustrated in Figure~\ref{fig:efb}. On the left hand-side, an
anti-parallelogram linkage is displayed. Quite similar to the hyperbolic case,
the moving revolute joints $B_{12}$, $B_{34}$ are obtained from the fixed
revolute joints $A_{12}$, $A_{34}$ by a reflection in the tangent of the conic
$\IC$. In fact, we may view this situation as limiting case of the hyperbolic
construction illustrated in Figure~\ref{fig:hfb-2f}. The homology center $P$ of
the hyperbolic case becomes a point at infinity in the Euclidean limit. Also
note that the connecting lines of $A_{12}$, $B_{34}$ and $A_{34}$, $B_{12}$ in
the anti-parallelogram are parallel while the parallelograms exhibit a well
known central similarity.

\begin{figure}
  \centering
  \includegraphics{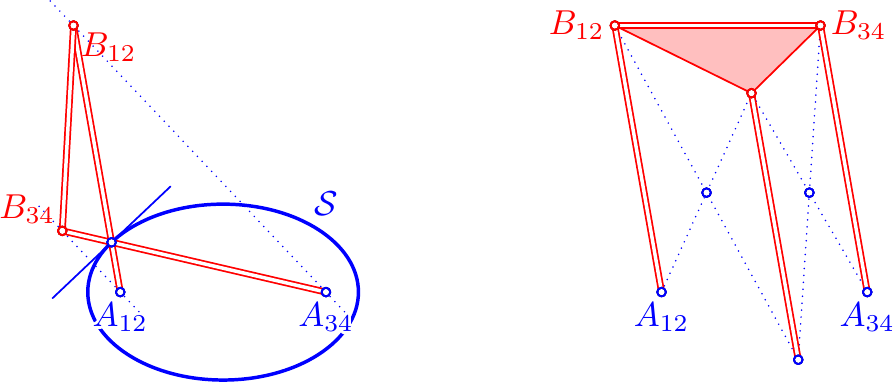}
  \caption{Euclidean anti-parallelogram and parallelogram mechanism}
  \label{fig:efb}
\end{figure}

The parallelogram linkage displayed in the right-hand side of
Figure~\ref{fig:efb} evades generic factorization theory of dual quaternions.
The infinitely many factorizations correspond to infinitely many legs that can
be added without disturbing the motion. The parallelogram linkage may be viewed
as a degenerate limit of the hyperbolic construction as explained in the next
paragraph.

In hyperbolic (and elliptic) geometry, it makes no sense to speak of
parallelograms: Given three generic points $A_{12}$, $A_{34}$, $B_{34}$ in the
hyperbolic or elliptic plane, we want to find a point $B_{12}$ such that
$\Quadrance{A_{12}}{A_{34}} = \Quadrance{B_{12}}{B_{34}}$ and
$\Quadrance{A_{12}}{B_{12}} = \Quadrance{A_{34}}{B_{34}}$. This problem has two
solutions $B^1_{12}$, $B^2_{12}$; their construction is shown in
Figure~\ref{fig:quad-construction}. There exist two mid-points $C_1$, $C_2$ of
$A_{12}$ and $B_{34}$ (see \cite{wildberger10}) and $B^\ell_{12}$ is the
reflection of $A_{34}$ in $C_\ell$ for $\ell \in \{1,2\}$. There is no way to
distinguish algebraically between $B^1_{12}$ and $B^2_{34}$ and in general
nothing special can be said about the corresponding motion modes and their
factorizations. In a suitable passage to the limit towards Euclidean geometry,
one midpoint tends towards infinity and leads to the Euclidean
anti-parallelogram case.

\begin{figure}
  \centering
  \includegraphics{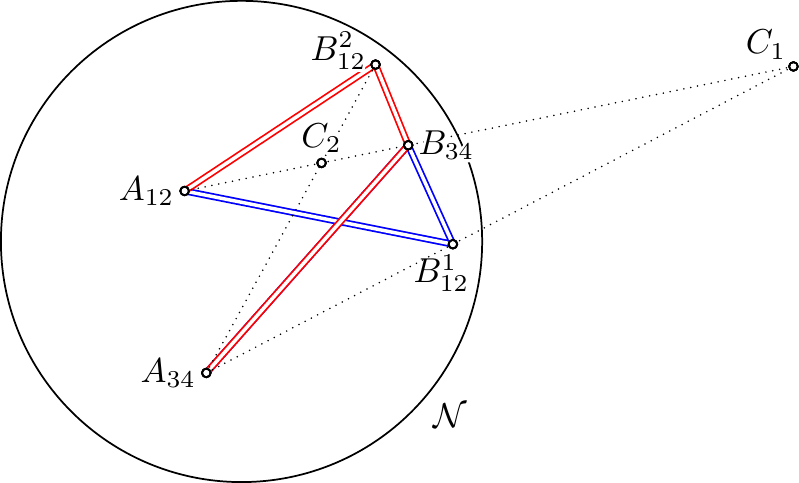}
  \caption{Construction of quadrilaterals with equal opposite quadrances}
  \label{fig:quad-construction}
\end{figure}

We want to conclude by emphasizing that we were only concerned with generic
cases in this paper. In particular, we always assumed that a finite number of
factorizations exist. There are monic quadratic polynomials $C \in \S[t]$ that
do not satisfy this assumption \cite{li18}. An investigation of their geometry
is a topic of future research.

\section*{Acknowledgment}

This work was supported by the Austrian Science Fund (FWF): P~31061 (The Algebra
of Motions in 3-Space).
 
\bibliographystyle{amsplain}

\begin{thebibliography}{10}

\bibitem{gallet17}
Matteo Gallet, Christoph Koutschan, Zijia Li, Georg Georg~Regensburger, and
  Nelly Schicho, Josef~Villamizar, \emph{Planar linkages following a prescribed
  motion}, Math. Comp. \textbf{86} (2017).

\bibitem{hegedus13}
G\'abor Heged\"us, Josef Schicho, and Hans-Peter Schr\"ocker,
  \emph{Factorization of rational curves in the {Study} quadric and revolute
  linkages}, Mech. Machine Theory \textbf{69} (2013), no.~1, 142--152.

\bibitem{inalcik17}
Abdullah Inalcik, Soley Ersoy, and Hellmuth Stachel, \emph{On instantaneous
  invariants of hyperbolic planes}, Math. Mech. Solids \textbf{22} (2017),
  no.~5, 1047--1057.

\bibitem{li15}
Zijia Li, Tudor-Dan Rad, Josef Schicho, and Hans-Peter Schr\"ocker,
  \emph{Factorization of rational motions: A survey with examples and
  applications}, Proceedings of the 14th IFToMM World Congress (Shuo-Hung
  Chang, ed.), 2015.

\bibitem{li15b}
Zijia Li, Josef Schicho, and Hans-Peter Schr\"ocker, \emph{Factorization of
  motion polynomials}, Accepted for publication in J. Symbolic Comp., 2018.

\bibitem{li18}
Zijia Li and Hans-Peter Schr\"ocker, \emph{Factorization of left polynomials in
  {Clifford} algebras: {State} of the art, applications, and open questions},
  Submitted for publication, 2018.

\bibitem{schoger96}
Astrid~Ulrike Schoger, \emph{{Koppelkurven und Mittelpunktskegelschnitte in der
  hyperbolischen Ebene}}, J. Geom. \textbf{57} (1996), no.~1--2, 160--176.

\bibitem{wildberger10}
Norman Wildberger, \emph{Universal hyperbolic geometry {II: A Pictorial
  Overview}}, KoG \textbf{14} (2010), 3--24.

\bibitem{wildberger11}
\bysame, \emph{Universal hyperbolic geometry {III: First Steps in Projective
  Triangle Geometry}}, KoG \textbf{15} (2011), 25--49.

\bibitem{wildberger13}
\bysame, \emph{Universal hyperbolic geometry {I: Trigonometry}}, Geom. Dedicata
  \textbf{163} (2013), 215--274.

\end{thebibliography}
\providecommand{\bysame}{\leavevmode\hbox to3em{\hrulefill}\thinspace}
\providecommand{\MR}{\relax\ifhmode\unskip\space\fi MR }
\providecommand{\MRhref}[2]{  \href{http://www.ams.org/mathscinet-getitem?mr=#1}{#2}
}
\providecommand{\href}[2]{#2}

\end{document}